\title{Packing a randomly edge-colored random graph with rainbow $k$-outs}
\author{Asaf Ferber\thanks{Institute of Theoretical Computer Science
ETH, 8092 Z\"urich, Switzerland. Email: asaf.ferber@inf.ethz.ch.} \and Gal
Kronenberg\thanks{School of Mathematical Sciences, Raymond and Beverly Sackler Faculty of Exact Sciences, Tel Aviv University, Tel Aviv, 6997801, Israel. Email: galkrone@mail.tau.ac.il.} \and
Frank Mousset\thanks{Institute of Theoretical Computer Science
ETH, 8092 Z\"urich, Switzerland. Email: frank.mousset@inf.ethz.ch.} \and Clara Shikhelman\thanks{Department of Mathematics, Tel Aviv University, Tel Aviv 6997801, Israel. Email: clara.shikhelman@gmail.com. Research
supported in part by an ISF grant.}}

\documentclass[a4paper,11pt]{article}
\usepackage{amsmath,amssymb,amsthm,color,epsfig,enumerate} 
\usepackage{verbatim}
\usepackage{algorithm2e}

\newtheorem{theorem}{Theorem}[section]
\newtheorem{lemma}[theorem]{Lemma}

\newtheorem{corollary}[theorem]{Corollary}

\newtheorem{remark}[theorem]{Remark}
\newtheorem{conjecture}[theorem]{Conjecture}

\newtheorem{question}[theorem]{Question}

\renewenvironment{proof}[1][\proofname]{{\noindent\bfseries #1. }}{\hfill$\Box$}

\newcommand{\eps}{\varepsilon}

\DeclareMathOperator{\Bin}{Bin}
\DeclareMathOperator{\Ber}{Ber}

\begin{document}
\setlength{\parskip}{1ex plus 0.5ex minus 0.2ex}
 \maketitle

\begin{abstract}
  Let $G$ be a graph on $n$ vertices and let $k$ be a fixed positive integer. We denote by $\mathcal G_{\text{$k$-out}}(G)$ the probability space consisting of subgraphs of $G$ where each vertex $v\in V(G)$ randomly picks $k$ neighbors from $G$, independently from all other vertices. We show that if $\delta(G)=\omega(\log n)$ and $k\geq 2$, then the following holds for every $p=\omega\left(\log n/\delta(G)\right)$. Let $H$ be a random graph obtained by keeping each $e\in E(G)$ with probability $p$ independently at random and then coloring its edges independently and uniformly at random with elements from the set $[kn]$. Then, w.h.p.\ $H$ contains $t:=(1-o(1))\delta(G)p/(2k)$ edge-disjoint graphs $H_1,\ldots,H_t$ such that each of the $H_i$ is \emph{rainbow} (that is, all the edges are colored with distinct colors), and such that for every monotone increasing property of graphs $\mathcal P$ and for every $1\leq i\leq t$ we have $\Pr[\mathcal G_{\text{$k$-out}}(G)\models \mathcal P]\leq \Pr[H_i\models \mathcal P]+n^{-\omega(1)}$. Note that since (in this case) a typical member of $\mathcal G_{\text{$k$-out}}(G)$ has average degree roughly $2k$, this result is asymptotically best possible. We present several applications of this; for example, we use this result to prove that for $p=\omega(\log n/n)$ and $c=23n$, a graph $H\sim \mathcal G_{c}(K_n,p)$ w.h.p.\ contains $(1-o(1))np/46$ edge-disjoint rainbow Hamilton cycles. More generally, using a recent result of Frieze and Johansson, the same method allows us to
prove that if $G$ has minimum degree $\delta(G)\geq (1+\varepsilon)n/2$, then there exist functions $c=O(n)$ and $t=\Theta(np)$ (depending on $\varepsilon$) such that the random subgraph $H\sim \mathcal G_{c}(G,p)$ w.h.p.\ contains $t$ edge-disjoint rainbow Hamilton cycles.
\end{abstract}

\section{Introduction}

In this paper we consider the following model of edge-colored random graphs.
Let $G$ be a graph, let $p\in [0,1]$ and let $c$ be a positive integer.
Then we define $\mathcal G_c(G,p)$ to be the probability
space of edge-colored graphs obtained by first choosing each edge of $G$ independently with probability $p$
 and then coloring each chosen edge
independently and uniformly at random with a color from the set
$[c]:=\{1,\ldots,c\}$. In the special case where $c=1$, we write $\mathcal G(G,p) := \mathcal G_c(G,p)$, and observe that this is just the standard binomial random graph model.
Furthermore, we use the abbreviations $\mathcal G_{c}(n,p)=\mathcal G_c(K_n,p)$
and $\mathcal G(n,p) = \mathcal G_1(n,p)$.
For $H\sim \mathcal G_c(G,p)$ and a graph $S$, we say that $H$ contains a \emph{rainbow} copy of $S$ if
$H$ contains a subgraph isomorphic to $S$ with all edges in distinct colors. The general theme
of much recent research has been to determine the conditions on $p$ and $c$ under which
a random graph $H\sim \mathcal G_c(n,p)$ contains, with high probability, a rainbow copy of a given graph $S$.

In~\cite{cooper2002multi}, Cooper and Frieze showed that for $p\approx 42\log n /n$ and $c=21n$, a graph
$H\sim \mathcal G_c(n,p)$ typically contains a rainbow Hamilton cycle. Later on, Frieze and Loh~\cite{frieze2014rainbow} improved this to
$p=\frac{\log n+\log\log n +\omega(1)}{n}$ and $c =(1+o(1))n$, which is asymptotically optimal with respect to both of the parameters $p$ and $c$.  Recently, Bal and Frieze \cite{bal2013rainbow} obtained the optimal $c$ by showing that
for $p=\omega(\log n/n)$ and $c = n/2$ (respectively $c=n$), a typical graph $H\sim \mathcal G_c(n,p)$
contains a rainbow perfect matching (respectively a rainbow Hamilton cycle). For general graphs,
Ferber, Nenadov and Peter \cite{ferber2013universality} showed that for every graph $S$ on $n$ vertices
with  maximum degree $\Delta(S)$ and for $c=(1+o(1))e(S)$, a typical $H\sim G_c(n,p)$ contains a rainbow
copy of $S$, provided that $p= n^{-1/\Delta(S)} \text{polylog}(n)$ (here, as elsewhere, $e(S)$ denotes the the number of edges in $S$). In this case, the number of colors $c$ is asymptotically optimal, whereas the edge probability $p$ is not.

Note that the results cited above consider the problem of finding a single rainbow copy of
some graph $S$ in a graph $H\sim \mathcal G_c(n,p)$.
A natural variant of this problem is the following, referred to as \emph{the problem of packing}:

\begin{question}\label{question1}
 Given a graph $S$ on $n$ vertices, for which values of $p,c,t$ does a graph $G\sim \mathcal G_c(n,p)$ typically contain $t$ edge-disjoint rainbow copies of $S$?
\end{question}

One could also pose a variant of Question~\ref{question1} in which one does not require the copies of $S$ to be rainbow. In this case, the question reduces to the following: for which $p$ and $t$ does the random graph $\mathcal G(n,p)$ contain $t$ edge-disjoint copies of $S$?
This problem is widely studied in the literature, and it is solved completely for the important case where $S$ is a Hamilton cycle. Indeed, it follows from~\cite{knox2013edge,krivelevich2012optimal,kuhn2013hamilton,kuhn2014hamilton} that for every $p=p(n)\in [0,1]$, a typical $H\sim \mathcal G(n,p)$ contains $t=\lfloor\delta(H)/2\rfloor$
edge-disjoint Hamilton cycles, which is clearly optimal.

In this paper we initiate the study of the rainbow variant of Question \ref{question1} and present a general method for tackling it. Roughly speaking, we show that if the graph $G$ satisfies a minimum degree condition, then in a typical $H\sim \mathcal G_c(G,p)$ one can find many edge-disjoint rainbow subgraphs, each of which is sampled from the well-studied random graph model $\mathcal G_{k\text{-out}}(G)$, which we define below. As an application, we use this method to show a first result related to finding ``many'' edge-disjoint rainbow Hamilton cycles in $\mathcal G_c(n,p)$ for certain values of $c$ and $p$. 

Let us now define the random graph model $\mathcal G_{k\text{-out}}(G)$,
where $G$ is a graph with minimum degree $\delta(G)$ and where $k\leq \delta(G)$ is a positive integer.
Given $G$ and $k$, we define $\mathcal G_{k\text{-out}}(G)$ to be the distribution of
subgraphs $H$ of $G$ obtained by the
following procedure:
each vertex $v\in V(G)$ chooses $k$ out-neighbors
uniformly at random among its neighbors in $G$ to create a digraph $D$;
then, $H$ is obtained by ignoring orientations and multiple edges in $D$.

This model was introduced by Walkup in 1980~\cite{walkup1980matchings},
where he proved that for every sufficiently large integer $n$, a graph
$H\sim\mathcal G_{2\text{-out}}(K_{n,n})$ typically contains a perfect matching.
This result is quite efficient in the sense that a typical graph in
$\mathcal G_{\text{2-out}}(K_{n,n})$ has only roughly $4n$ edges.
In contrast, the standard random graph model $\mathcal G(K_{n,n},p)$, obtained by picking each
edge of $K_{n,n}$ independently with probability $p$, starts having a perfect matching only at
$p\approx \log n/n$ (see~\cite{bollobas1998random}), where it has roughly $n\log n$ edges.
The reason for this is that if $p=o(\log n / n)$, then a typical graph in $\mathcal G(K_{n,n},p)$
contains isolated vertices, which clearly precludes the existence of a perfect matching; on the other hand,
the model $\mathcal G_{\text{2-out}}$ makes sure that there are no isolated vertices from the start.
As we will see, the fact that the model $\mathcal G_{k\text{-out}}$ makes such efficient use of
the available edges makes it especially useful for packing applications.

The model $\mathcal G_{k\text{-out}}(G)$
has attracted a lot of attention from various researchers in the last couple of decades
(see~\cite{bohman2009hamilton,cooper1994hamilton,fenner1983existence,frieze1986maximum,frieze1987hamiltonian,karonski2003existence}). Here we give a short overview of some related results. In~\cite{frieze1986maximum},
Frieze extended the above-mentioned result of Walkup to the complete graph $K_n$, by showing that a typical
graph in $\mathcal G_{2\text{-out}}(K_n)$
contains a perfect matching, provided that $n$ is even. This was further improved
by Karo{\'n}ski and Pittel, who showed that
$H\sim \mathcal G_{(1+e^{-1})\text{-out}}(K_{n,n})$ typically contains a perfect matching, where
$G_{(1+e^{-1})\text{-out}}(K_{n,n})$
is obtained by first picking a random element of $\mathcal G_{1\text{-out}}(K_{n,n})$ and then giving
every vertex that has not been chosen as a neighbor by another vertex a `second chance' to pick a second
random neighbor~\cite{karonski2003existence}.
In another direction, Fenner and Frieze showed
in~\cite{fenner1983existence} that a typical graph in $\mathcal G_{\text{23-out}}(K_n)$ contains a Hamilton cycle.
Recently, Bohman and Frieze proved that
a graph $H\sim \mathcal G_{\text{3-out}}(K_n)$ typically contains a Hamilton cycle~\cite{bohman2009hamilton}.
Note that this result is optimal, as one can show that a typical graph in
$\mathcal G_{\text{2-out}}(K_n)$ is not Hamiltonian (see~\cite{bohman2009hamilton}).
 The results that we have mentioned so far only treat
the cases where $G$ is either the complete graph or the complete bipartite graph. It is only very
recently that Frieze and Johansson~\cite{frieze2014random} proved several interesting results for arbitrary graphs $G$. Among
other things, they showed that for every $\varepsilon>0$ there exists a positive integer $k_{\varepsilon}$ such that for all graphs
$G$ with minimum degree $\delta(G)\geq (1+\varepsilon)n/2$, a typical graph in $\mathcal G_{k_{\varepsilon}\text{-out}}(G)$
is Hamiltonian.

The following theorem is the main result of our paper.
In this theorem we show that one can find $(1-o(1))\frac{\delta(G)p}{2k}$ edge-disjoint rainbow subgraphs in a typical $H\sim \mathcal G_{kv(G)}(G,p)$,
each of which is distributed {\em almost} as $\mathcal G_{k\text{-out}}(G)$.
Note that since a typical member of $\mathcal G_{k\text{-out}}(G)$ has average degree roughly $2k$,
this result is asymptotically optimal.

\begin{theorem}\label{thm:main} Let $\varepsilon>0$ be a constant, let $k\geq 2$ be an integer and let $\mathcal{P}$ be a monotone increasing graph property. Let $G$ be a graph on $n$ vertices with minimum
degree $\delta(G)=\omega(\log n)$, and let $p\in (0,1]$ be such that $p=\omega\left(\frac{\log n}{\delta(G)}\right)$. Then,
w.h.p.\ a graph $H\sim \mathcal G_{kn}(G,p)$ can be generated as $H=H_0\cup H_1\cup\dots\cup H_t$ such that the following holds.
  \begin{enumerate}
    [$(i)$]
    \item $H_0,H_1,\ldots,H_t$ are edge-disjoint subgraphs,
    \item $t:=\left(1-\varepsilon\right)\frac{\delta(G) p}{2k}$,
   	\item for every $1\leq i\leq t$, $E(H_i)$ is rainbow, and
    \item for every $1\leq i\leq t$, $\Pr\left[\mathcal G_{k\text{-out}}(G)\models \mathcal{P}\right]\leq \Pr\left[H_i\models \mathcal{P}\right]+n^{-\omega(1)}$.

  \end{enumerate}
\end{theorem}


\begin{remark}\label{re:main}
Observe that if $\Pr[\mathcal G_{\text{$k$-out}}(G)\models \mathcal P]\geq 1-o(1/t),$ then from Theorem \ref{thm:main} w.h.p.\ $H$ contains $t$ edge-disjoint subgraphs $H_1,\ldots,H_t$, each of which satisfies $\mathcal P$  and $E(H_i)$ is rainbow for each $1\leq i\leq t$.
\end{remark}

Remark~\ref{re:main} is especially interesting in light of the above-mentioned
embedding results for $\mathcal G_{k\text{-out}}(G)$.
Essentially, for every such embedding result, we can get an analogous rainbow packing result `for free,' although
it is necessary to show that the embedding succeeds with sufficiently large probability.
In the following, we show some examples of results that can be obtained with this method. The (short) proofs can be found in Section~\ref{sec:apps}.
Our first application concerns the problem of finding many edge-disjoint rainbow perfect matchings in a
random bipartite graph.
 
\begin{theorem}\label{thm:PM}
  Let $n$ be a positive integer, let $c=3n$ and let $p=\omega\left(\frac{\log n}{n}\right)$. Then w.h.p.\ a graph $G\sim\mathcal G_c(K_{n,n},p)$
  contains at least $t=(1-o(1))\frac{np}{6}$ edge-disjoint rainbow perfect matchings.
\end{theorem}

As a second application we show how to find many edge-disjoint rainbow  Hamilton cycles in a typical $\mathcal{G}_c(n,p)$. 

\begin{theorem}\label{thm:Ham}
Let $n$ be a positive integer and $c=23n$ and let $p =\omega\left(\frac{\log n}{n}\right)$. Then w.h.p.\ a graph $G \sim \mathcal{G}_{c}(n,p)$ contains at least $t=(1-o(1))\frac{np}{46}$ edge-disjoint rainbow Hamilton cycles.
\end{theorem}

Note that in Theorems \ref{thm:PM} and \ref{thm:Ham}, the parameters $c$ and $t$ are almost certainly not optimal.  We conjecture that the following is correct.

\begin{conjecture}
	Let $n$ be a positive integer and let  $p=\omega\left(\frac{\log n}{n}\right) $. There exists $c=(1+o(1))n$ such that w.h.p.\ a graph $H\sim \mathcal{G}_c(K_{n,n},p)$ contains $(1-o(1)){np}$ edge-disjoint rainbow perfect matchings.
\end{conjecture}

\begin{conjecture}
Let $n$ be a positive integer and let  $p=\omega\left(\frac{\log n}{n}\right)$. There exists $c=(1+o(1))n$ such that
w.h.p.\ a graph $H\sim \mathcal{G}_c(n,p)$ contains $(1-o(1))\frac{np}2$ edge-disjoint rainbow Hamilton cycles.
\end{conjecture}

Finally, using the above-mentioned result of Frieze and Johansson, we can also prove the following
  variant of Theorem~\ref{thm:Ham} that applies to graphs satisfying a minimum degree condition.
\begin{theorem}\label{thm:Dirac-Ham}
Let $\varepsilon >0$,  let
$G$ be a graph on $n$ vertices with minimum degree $\delta(G)\geq (1+\varepsilon)n/2$, and let
$p=\omega(\log n/n)$. Then there exist 
functions $c=O(n)$ and $t=\Theta(np)$ such that w.h.p.\ a graph $H \sim \mathcal{G}_{c}(G,p)$ contains at least $t$ edge-disjoint rainbow Hamilton cycles.
\end{theorem}

\subsection{Notation and Terminology}\label{sec:notation}
Our graph-theoretic notation is standard and follows that of \cite{West}.
Given a graph $G$, we denote the vertex and edge sets of $G$ by $V=V(G)$ and $E=E(G)$,
respectively. We write $e(G)=|E(G)|$ and $v(G)=|V(G)|$. For disjoint
sets $X,Y\subseteq V(G)$,
we let $e_G(X,Y)$ denote the number of edges $\{x,y\}\in E(G)$ such that $x\in X$ and $y\in Y$. The minimum degree of $G$ is denoted
by $\delta(G)$. If $\mathcal P$ is a graph property, then we write $G\models \mathcal P$ to mean
that $G$ satisfies the property $\mathcal P$. For a probability space $\mathcal G$, we sometimes write
$\mathcal G \models \mathcal P$ for the event that the random element of $\mathcal G$ satisfies $\mathcal P$
(i.e., if $H\sim \mathcal G$, then $\Pr[\mathcal G \models \mathcal P]=\Pr[H\models\mathcal P]$).

An \emph{oriented} graph $G$ consists of a
  set of vertices $V(G)$, and a set of \emph{arcs} (or \emph{oriented}
  edges) $E(G)$ which contains elements of the form $(x,y)\in
  V(G)\times V(G)$, where it is not allowed to have both $(u,v)$ and $(v,u)$ in $E(G)$. We usually write $\overrightarrow{xy}$ to denote the arc $(x,y)$. For an oriented graph $G$ and a
  vertex $x\in V(G)$ we let $N^+_G(x)=\{y\in V(G)~|~ \overrightarrow{xy}\in E(G)\}$ be the
  \emph{out-neighborhood} and
  $d^+_G(x)=|N^+_G(x)|$ the \emph{out-degree} of $x$.

If $G$ is an edge-colored graph, then we call a subset
$F\subseteq E(G)$ \emph{rainbow} if all the elements of
$F$ are colored in distinct colors. For a given graph $H$ we say
that $G$ contains a \emph{rainbow copy of $H$} if and only if it
contains a copy $H'$ of $H$ such that $E(H')$ is rainbow.

If $B$ is a bipartite graph with partite sets $X$ and $Y$ of respective sizes $n$ and $kn$ (for a positive
integer $k$), then
a \emph{perfect $k$-matching} refers to a set of $n$ vertex-disjoint $k$-stars in $B$ with the central vertex in $X$.

In addition, we make use of several random graph models, some of which are well-known and some of which are new.
In the following, we will define the different models that appear in this paper.

Given $G$, we denote by $\mathcal G(G,p)$
the probability space of random subgraphs of $G$ obtained by
retaining each edge of $G$ with probability $p$, independently at
random. For the special cases where $G=K_n$ or $G=K_{n,m}$ we write
$\mathcal G(n,p):=\mathcal G(K_n,p)$ and $\mathcal
B(n,m,p):=\mathcal G(K_{n,m},p)$, respectively.
For every positive integer $c$, we define $\mathcal G_c(G,p)$ to be the probability
space of the edge-colored graphs obtained by first choosing a random
element from $\mathcal G(G,p)$ and then coloring each edge
independently and uniformly at random with a color from the set
$[c]:=\{1,\ldots,c\}$.

Let $G$ be a graph with minimum degree $\delta(G)$ and let $k\leq
\delta(G)$ be a positive integer. Then we let $\mathcal
G_{k\text{-out}}(G)$ be the probability space of subgraphs $H$ of $G$ obtained by the
following procedure:
each vertex $v\in V(G)$ independently chooses $k$ random out-neighbors
among its neighbors in $G$ to create the random digraph $H'$.
Then $H$ is obtained by ignoring orientations in $H'$. This model first introduced and studied by Walkup in \cite{walkup1980matchings} and later on studied by Frieze in \cite{frieze1986maximum}.
We use the abbreviation $\mathcal G_{k\text{-out}}(n):=\mathcal G_{k\text{-out}}(K_n)$.

Let $\mathcal G_{k\text{-out}}^*(G)$ be the probability space
consisting of subgraphs $H$ of $G$ obtained as follows. First we
create an oriented graph $D$ by orienting the edges of $G$ uniformly
at random. Then each vertex $v\in V(G)$ picks
$t_v=\min\{k,d_D^+(v)\}$ out-edges $e^v_1,\ldots,e^v_{t_v}$,
independently at random. We then create $H'$ by setting
$E(H')=\{e^v_i\mid v\in V(G) \text{ and } 1\leq i\leq t_v\}$.
Finally, we obtain $H$ by ignoring the orientations in $H'$.
Perhaps unsurprisingly, the model $\mathcal G^*_{k\text{-out}}(G)$ is closely related
to the more standard $\mathcal G_{k\text{-out}}(G)$.
The relationship between the two models will be discussed in Section~\ref{Section:coupling}.

Lastly, we define $\mathcal
B^{\ell}_{\text{$k$-out}}(a,b)$ (``left'' $k$-out) to be the probability space of all bipartite graphs
with vertex set $A\cup B$, where $|A|=a$ and $|B|=b$, and where each
vertex in $A$ claims randomly $k$ neighbors from $B$.

\section{Tools and auxiliary results}

In this section we present some tools and auxiliary results which will be used in the proof of our main result.

\subsection{Probabilistic tools}
We use extensively the following well known bound on the lower and
the upper tails of the Binomial distribution due to Chernoff (see,
e.g., \cite{AloSpe2008}):

\begin{lemma}[Chernoff]\label{lemma:chernoff}
Let $X_1,\ldots, X_n$ be independent random variables,
$X_i\in\{0,1\}$ for each $i$. Let $X=\sum_{i=1}^nX_i$ and write
$\mu=\mathbb{E}(X)$, then
\begin{enumerate}[(i)]
    \item $\Pr\left(X\le(1-a)\mu\right)\le\exp\left(-\frac{a^2\mu}{2}\right)$ for every $a>0.$
    \item $\Pr\left(X\ge(1+a)\mu\right)\le\exp\left(-\frac{a^2\mu}{3}\right)$ for every $0 < a < 1.$
\end{enumerate}
\end{lemma}


For the proof of Lemma~\ref{lemma:concentration} below, we will also need the following
concentration inequality of Talagrand (see~\cite{MR}). 

\begin{lemma}[Talagrand]\label{lemma:talagrand}
  Let $X$ be a non-negative random variable, not identically $0$, which is
  determined by $n$ independent trials $T_1,\dotsc,T_n$, and satisfying the following
  for some $c,r>0$:
  \begin{enumerate}[(a)]
  \item (Lipschitz condition) changing the outcome of any one trial can affect $X$ by at most $c$, and
  \item (Certifiability) for every $s$, if $X\geq s$, then there is a set of at most $rs$ trials whose outcomes
    certify that $X\geq s$;
  \end{enumerate}
  then for every $0\leq t \leq \mathbb E[X]$, we have
  \[ \Pr\left[|X-\mathbb E[ X]|> t + 60c\sqrt{r\mathbb E[X]}\right]\leq 4e^{-\frac{t^2}{8c^2r\mathbb E[X]}}.\]
\end{lemma}

\subsection{Matchings in graphs}
The following two lemmas and corollary guarantee the existence of many edge-disjoint
perfect $k$-matchings (from `left to right', i.e., each star has its center in the part of size $n$) in $\mathcal B^{\ell}_{s(n)\text{-out}}(n,kn)$, where $s(n)=\omega(\log{n})$.
In the proof, we will make use of the following lemma due to Gale
and Ryser~\cite{lovaszproblems}. 
\begin{lemma}[Gale-Ryser]\label{lemma:gale-ryser}
  A bipartite graph $G=(A\cup B,E)$ with $|A|=|B|$ contains an $r$-factor if and only if
  for all $X\subseteq A$ and $Y\subseteq B$, the following holds:
  \[e_G(X,Y)\geq r(|X|+|Y|-|B|)\text.\]
\end{lemma}

\begin{lemma}\label{lemma:many-matchings}
  Let $p=\omega(\log{n}/n)$.
  Then for every $\varepsilon>0$, w.h.p.\ a graph $G\sim\mathcal B(n,n,p)$ contains a family of $(1-\varepsilon)np$
  edge-disjoint perfect matchings.
\end{lemma}

\begin{proof}
  First, note that every $r$-regular bipartite graph contains $r$ edge-disjoint perfect matchings, as can be easily seen
  by a repeated use of Hall's condition. Therefore, it is enough to show that w.h.p.\ $G$ contains an $r$-factor,
  where $r=(1-\varepsilon)np$.

  Denote by $A$ and $B$ the parts of $G$.
  By Lemma~\ref{lemma:gale-ryser}, it is enough to show that w.h.p.\ we have
  \begin{equation}\label{eq:condition}e_{G}(X,Y)\geq r(|X|+|Y|-n)\end{equation}
  for all $X\subseteq A$ and $Y\subseteq B$.
  We split the set of pairs $(X,Y)\in \mathcal{P}(A)\times\mathcal{P}(B)$
  into three sets $S_1$, $S_2$ and $S_3$, defined by
  \begin{align*} & S_1 = \{ (X,Y) \mid |X|+|Y|\le n\}\text,\\
    &S_2 = \{ (X,Y) \mid |X|\leq |Y|\}\setminus S_1\text{, and}\\
    &S_3 = \{ (X,Y) \mid |X|>|Y|\}\setminus S_1\text.
  \end{align*}

  Note that for all pairs $(X,Y)\in S_1$ we have $|X|+|Y|\le n$, and so
  \eqref{eq:condition} is trivially satisfied. Therefore, we may assume that $(X,Y)\not\in S_1$.
  Now, note that $e_G(X,Y)\sim \Bin(|X||Y|,p)$, so by the Chernoff bound (Lemma~\ref{lemma:chernoff}) we have
  \begin{equation}\label{eq:xychern} \Pr[e_G(X,Y)\leq (1-\varepsilon/2)|X||Y|\cdot p] \leq e^{-\varepsilon^2|X||Y|p/8}\text.\end{equation}

  For all $(X,Y)\in S_2$, we have by definition both $n\leq |X|+|Y|$ and $|X|\leq |Y|$, i.e.,
  $|Y|\geq \max{\{n-|X|,|X|\}}$. Thus, by \eqref{eq:xychern} and the union bound,
  the probability that there exists a pair $(X,Y)\in S_2$
  such that $e_G(X,Y)\leq (1-\varepsilon/2)|X||Y|\cdot p$ is at most
  \[ \sum_{x=1}^{n}\binom{n}{x}\sum_{y=\max{\{n-x,x\}}}^n\binom{n}{y}e^{-\varepsilon xyp/8}. \]
  Using that $p = \omega(\log n/n)$, the inner sum can be bounded from above by
  \[ \sum_{y=x}^n n^y e^{-\varepsilon xyp/8} \leq
  \frac{n^xe^{-\varepsilon x^2p/8}}{1-ne^{-\varepsilon xp/8}} = n^{-x\cdot\omega(1)}\]
  if $x\geq n/2$, and by
  \[ \sum_{y=n-x}^n n^{n-y} e^{-\varepsilon xyp/8}
  = n^n\sum_{y=n-x}^n n^{-y} e^{-\varepsilon xyp/8}\leq 
  n^n\frac{n^{x-n}e^{-\varepsilon x(n-x)p/8}}{1-n^{-1}e^{-\varepsilon xp/8}} = n^{-x\cdot\omega(1)}\]
  if $x< n/2$. Putting everything together, we obtain
  \[\Pr[\exists (X,Y)\in S_2\text{ with }e_{G}(X,Y)\leq (1-\varepsilon/2)|X||Y|\cdot p] \leq
  \sum_{x=1}^{n}\binom{n}{x} n^{-x\cdot \omega(1)} = o(1)\text.\]
  Now observe that, since $|X|$ and $|Y|$ are at most $n$, we have
  \[ |X||Y|-n(|X|+|Y|-n) = (n-|X|)(n-|Y|) \ge 0 \text,\]
  and therefore $|X||Y|\geq n(|X|+|Y|-n)$.
  Thus w.h.p.\ for all $(X,Y)\in S_2$
  \[ e_{G}(X,Y) \geq (1-\varepsilon/2)|X||Y|p \geq (1-\eps)np(|X|+|Y|-n)
  = r (|X|+|Y|-n)\text. \]
  In a similar way one can show that w.h.p.\ every $(X,Y)\in S_3$ satisfies \eqref{eq:condition}.
\end{proof}

\begin{corollary}\label{cor:star-factors}
  Let $k\geq 1$ be an integer and let $s:=s(n)\le kn$ be such that $s=\omega(\log n)$.
  Then for every $\varepsilon>0$, w.h.p.\ a graph $G\sim\mathcal B_{s\text{-out}}^\ell(n,kn)$ contains
  $(1-\varepsilon)\frac {s}{k}$ edge-disjoint perfect $k$-matchings.
\end{corollary}
\begin{proof}
  Let $\eps>0$ and let $p := (1-\eps/2)s/(kn)$.
  We will show that there exists a joint distribution of random
  graphs $\mathcal B(n,kn,p)$ and $\mathcal B_{s\text{-out}}^\ell(n,kn)$ such that
  w.h.p.\ we have $\mathcal B(n,kn,p)\subseteq \mathcal B_{s\text{-out}}^\ell(n,kn)$.
  Once we have proved this, the corollary will follow easily from Lemma~\ref{lemma:many-matchings}. Indeed, let $L$ and $R$ be the parts of $\mathcal B(n,kn,p)$ and
  $\mathcal B_{s\text{-out}}^\ell(n,kn)$ (which, for simplicity, we identify), where $|L|=n$ and $|R|=kn$,
  and let $R = R_1\cup \dotsb \cup R_k$ be a partitioning of $R$ into sets of size $n$. Now, by exposing
  the edges of $G\sim\mathcal B_{s\text{-out}}^\ell(n,kn)$, and making use of the above mentioned coupling, we obtain
  that for each $i\in[k]$, there exists a subgraph $H_i\subseteq G[L\cup R_i]$ distbuted as $\mathcal B(n,n,p)$.
  In order to complete
  the proof, we apply Lemma~\ref{lemma:many-matchings} to
  get that w.h.p.\ each
  $H_i$ contains a collection $(M_{i,1},\dotsc,M_{i,t})$ of
  edge-disjoint perfect matchings, where $t=(1-\eps/2)np\geq (1-\eps)s/k$. Finally, for each $j\in [t]$, by
  taking all the edges in
  $\bigcup _{i=1}^{k} M_{i,j}$, we obtain a perfect $k$-matching from $L$ to $R$.
  By the construction, it is clear that these $k$-matchings are edge-disjoint.

  It is thus sufficient to describe the joint distribution.
  For this aim, let $G\sim \mathcal B_{s\text{-out}}^\ell(n,kn)$
  and note that w.h.p.\ $G$ can be generated in the following way: let $H\sim \mathcal B(n,kn,p)$ and for
  each $x\in L$, let $d_x := s-d_H(x)$. Note that w.h.p.\ we have $d_x\geq 0$ for all $x\in L$
  (this can be easily obtained using Chernoff, the union bound, and the fact that $s=\omega(\log n)$).
  In order to generate $G$, take $H$ and for each vertex $x\in L$, choose $d_x$ additional neighbors from $R$
  independently and uniformly at random. This completes the proof of the corollary.
\end{proof}

\subsection{The number of multiplicities in a random multi-set}

Our final tool states that in a large random multi-subset of a large set, the number of elements
that occur with a given multiplicity is concentrated.

\begin{lemma}\label{lemma:concentration}
  Let $k\geq 2$ be an integer, let $\eps \in (0,1)$ and let $n\in\mathbb N$.
  Let $\alpha=\alpha(n) \in (0,1]$ be such that $\alpha(n) = \omega(\log n/ n)$. Let $C$ be
  a multi-subset of $[kn]$ of size exactly $\alpha n$, chosen uniformly at random.
  For every $r\in\mathbb N$, write $m_r$ for the number of elements of $[kn]$ that occur in $C$
  exactly $r$ times.
  Then there exists a constant $r_0=r_0(\eps,k)$ such that
  \begin{enumerate}[(i)]
  \item for every $r\in [r_0]$, we have $\mathbb E[m_r]=\omega(\log n)$,
  \item with probability $1-e^{-\omega(\log n)}$,
    we have $|m_r-\mathbb E[m_r]|\leq \eps \mathbb E[m_r]$ for every $r\in[r_0]$, and
  \item $\sum_{r\in [r_0]} r \mathbb E [m_r]\geq (1-\eps) \alpha n$.
  \end{enumerate}
\end{lemma}
\begin{proof} Note that it is enough to prove the lemma for $\varepsilon>0$ which is sufficiently small.
  For technical reasons, we need to distinguish between two cases.

  {\bf Case I:} $\alpha\leq \eps /2$.
  We will show that in this case, the choice $r_0=1$ satisfies {\em (i)-(iii)}. Let us start with
  proving {\em (i)} and {\em (iii)}. For each element $t\in [kn]$, the probability that $t$
  appears exactly once in $C$ is \[\binom{\alpha
    n}{1}\frac{1}{kn}\left(1-\frac{1}{kn}\right)^{\alpha n-1}
  \geq \left(1-\frac{\alpha-1/n}{k}\right)\alpha/k\geq
  (1-\eps/3)\alpha/k,\]
  hence $\mathbb{E}[m_1]\geq (1-\eps/3)\alpha n = \omega(\log n)$, which proves {\em (i)}.
  Since \[\sum_{r\in [r_0]}r\mathbb E[m_r] = \mathbb E[m_1]\geq (1-\eps/3)\alpha n,\] this also proves {\em (iii)}.

  Now let us prove {\em (ii)}. Since $\mathbb{E}[m_1]\geq (1-\eps/3)\alpha n$, we have
  $m_1\leq \alpha n \leq (1+\eps)\mathbb E[m_1]$ deterministically.
  It remains to prove that with probability $1-e^{-\omega(\log n)}$,
  we have $m_1\geq (1-\eps)\mathbb E[m_1]$. In fact, we prove the slightly stronger statement that
  $m_1\geq (1-\eps)\alpha n$.
  Note that we can choose the elements of $C$ one by one, in each
  step choosing an element of $[kn]$ uniformly at random. Since we pick $\alpha n$ elements in total,
  it is clear that in every step, the probability of drawing an element that has not been encountered before
  is at least $1-\alpha/k\geq1- \eps/2$. Thus we can bound from below $m_1$
  using a random variable $X\sim \Bin(\alpha n,1-\eps/2)$.
  Using the Chernoff bounds (Lemma~\ref{lemma:chernoff}), we have
  \[ \begin{split}
    \Pr[m_1< (1-\eps)\mathbb \alpha n] &\leq \Pr[X< (1-\eps)\alpha n]\\
    &\leq \Pr\Big[X< (1-\eps^2)\cdot \mathbb E[X]\Big]\\
    &\leq e^{-\eps^4(1-\eps/2)\alpha n/3} = e^{-\omega(\log n)}.
  \end{split} \]
  This proves {\em (ii)}.

  {\bf Case II:} $\alpha\geq \eps/2$. Let $r_0=r_0(\eps,k)$ be the
  smallest positive integer for which
  \begin{equation}\label{eq:defr0} 2 k^{-r_0} \leq \eps^2/2\leq \eps \alpha. \end{equation}
  We will show that this choice of $r_0$ satisfies statements $(i)$-$(iii)$ of the lemma.

  During the proof, we make use of the random variables $m_{\geq r} := \sum_{s\geq r}m_s$, for
  $r\in \mathbb N$.
  First, let us estimate the expectations of the variables $m_r$ and $m_{\geq r}$.
  The probability that a fixed element of $[kn]$ occurs with multiplicity
  $r$ is $\binom{\alpha n}{r}(kn)^{-r}(1-(kn)^{-1})^{\alpha n-r}$. Therefore, we have
  \begin{equation}\label{eq:m_r}\mathbb E[m_{r}] = kn\binom{\alpha n}{r}(kn)^{-r}(1-(kn)^{-1})^{\alpha n-r}
    \geq \frac12 kn\binom{\alpha n}{r}(kn)^{-r},\end{equation}
  if $n$ is large enough, where we used that $\alpha \leq 1$ and $k\geq 2$.
  Observe that this implies that for every $r\in [r_0+1]$, we have
  \begin{equation}\label{eq:m_r_lower} \mathbb E[m_{r}] \geq \frac{kn}2 \left(\frac{\alpha n}{r}\right)^{r}(kn)^{-r}
  \geq \frac{kn \eps ^r}{2^{r+1}r^rk^r} = \Omega(n),\end{equation} which already proves $(i)$.
  In order to estimate $\mathbb E[m_{\geq r}]$, note that
  \[ \sum_{s\geq r}\binom{\alpha n}{s}(kn)^{-s}(1-(kn)^{-1})^{\alpha n-s} = \Pr\left[\Bin(\alpha n, (kn)^{-1})
  \geq r\right] \]
  and therefore is trivially upper bounded by $\binom{\alpha n}{r} (kn)^{-r}$. Thus, we get
  \begin{equation}\label{eq:m_gr}\begin{split}\mathbb E[m_{\geq r}]=
      \sum_{s\geq r}\mathbb E[m_s]
      &= kn \sum_{s\geq r}\binom{\alpha n}{s}(kn)^{-s}(1-(kn)^{-1})^{\alpha n-s}\\
      &\leq kn\binom{\alpha n}{r} (kn)^{-r}.
  \end{split}
  \end{equation}

  We will now prove {\em (ii)}, i.e.\, that for each $r\in [r_0]$, $m_r$ is concentrated around its expectation.
  For this aim, we proceed as follows.
  We start with showing that the random variables $m_{\geq r}$ are concentrated around their
  expectations for all $r\in [r_0+1]$. Then we show that $\mathbb E[m_{\geq r}]$ is rapidly
  decreasing in $r$. Lastly,
  using the fact that $m_r = m_{\geq r}-m_{\geq r+1}$, we obtain
  the concentration result for $m_r$.

  For showing the concentration of $m_{\geq r}$, we make use of Talagrand's inequality
  (Lemma~\ref{lemma:talagrand}), which requires us to write
  $m_{\geq r}$ as a function of independent trials $T_i$ such that the following two conditions hold:
  \begin{enumerate}[(a)]
  \item \emph{(Lipschitz condition)} changing the outcome of any one trial can affect $m_{\geq r}$ by at most $1$, and
  \item \emph{(Certifiability)}
    for every $s$, if the outcome of $m_{\geq r}$ is at least $s$,
    then there is a set of at most $rs$ trials whose outcomes
    certify this.
  \end{enumerate}
  To do so, let $\{T_i\mid 1\leq i \leq \alpha n\}$ be a family of mutually independent random
  variables distributed uniformly on $[kn]$. Then we can define $C$ by collecting
  the outcomes of the variables $T_i$, and it is easy to see that
  the variables $m_{\geq r}$ are completely determined by the trials $T_1,\dotsc,T_{\alpha n}$.
  Moreover, it is clear that conditions (a) and (b) are satisfied.
  Therefore, we can apply Talagrand's inequality (Lemma~\ref{lemma:talagrand}) to the random variable $m_{\geq r}$.
  For every $r\in [r_0+1]$ and for all large enough $n$, we thus get
  \[ \begin{split}
    & \Pr\Big[|m_{\geq r}-\mathbb E[m_{\geq r}]|\geq \frac{\eps}{3} \mathbb E[m_{\geq r}]\Big]\\
    & \leq \Pr\Big[|m_{\geq r}-\mathbb E[m_{\geq r}]|> \frac{\eps}{4} \mathbb E[m_{\geq r}] +
    60\sqrt{r\mathbb E[m_{\geq r}]}\Big]\\
    & \leq 4e^{-\eps^2 \mathbb E[m_{\geq r}]/128r}\\
    & = e^{-\Omega(n)},
  \end{split} \]
  where we used that by \eqref{eq:m_r_lower}, for all $r\in [r_0+1]$, we have
  $\mathbb E[m_{\geq r}] \geq \mathbb E[m_r]= \Omega(n)$.
  By taking the union bound over constantly many events, this gives
  \[ \Pr\left[\forall r\in [r_0+1]: m_{\geq r} \in (1\pm \eps/3) \mathbb E[m_{\geq r}]\right] = 1-e^{-\omega(\log n)}\text. \]
  Now note that from \eqref{eq:m_r} and \eqref{eq:m_gr}, we get
  $\mathbb E[m_{\geq r}]\leq 2\mathbb E[m_r]$, which implies that
  \[ \mathbb E[m_{\geq r+1}] =\mathbb E[m_{\geq r}]-\mathbb E[m_r] \leq \mathbb E[m_r].\]
  Assume for now that $m_{\geq r} \in (1\pm \eps/3) \mathbb E[m_{\geq r}]$ holds for all $r\in [r_0+1]$;
  as we have seen, this is the case with probability $1-e^{-\Omega(n)}$.
  Then from $\mathbb E[m_{\geq r+1}] \leq \mathbb E[m_r]$, we get
  \[ \begin{split}
    m_r & = m_{\geq r}-m_{\geq r+1}\\
    & \geq (1-\eps/3)\mathbb E[m_{\geq r}] - (1+\eps/3) \mathbb E[m_{\geq r+1}]\\
    & = (1-\eps/3)(\mathbb E[m_{\geq r}]-\mathbb E[m_{\geq r+1}])-2\eps \mathbb E[m_{\geq r+1}]/3\\
    & \geq (1-\eps/3)\mathbb E[m_r]-2\eps \mathbb E[m_{r}]/3\\
    & = (1-\eps)\mathbb E[m_r]
  \end{split} \]
  and, similarly, $m_r\leq (1+\eps)\mathbb E[m_r]$.
  This proves {\em (ii)}.

  Finally, we will prove {\em (iii)}. Note that since $\sum_{r\geq 1}r\mathbb E[m_r]=\alpha n$,
  {\em (iii)} is equivalent to
  \begin{equation}\label{eq:thing1} \sum_{r>r_0}r\mathbb E[m_r]\leq \eps \alpha n.\end{equation}
  Now, by \eqref{eq:m_r} we have
  \[ \begin{split}
    \frac{\mathbb E[m_{r}]}{\mathbb E[m_{r+1}]}
    &= \frac{kn \binom{\alpha n}{r}(kn)^{-r}(1-(kn)^{-1})^{\alpha n-r}}{kn\binom{\alpha n}{r+1}(kn)^{-r-1}
      (1-(kn)^{-1})^{\alpha n -r -1}}\\
    &=\frac{kn (r+1)(1-(kn)^{-1})}{\alpha n}\\
    &\geq \frac12\cdot\frac{k}{\alpha}(r+1)
    \text,
  \end{split} \]
  for every $r\in\mathbb N$.
  It follows that $(r+1) \mathbb E[m_{r+1}]\leq 2(\alpha/k)\mathbb E[m_r]$.
  From this, using \eqref{eq:m_gr} and the definition \eqref{eq:defr0} of $r_0$, we get
  \[ \begin{split}
    \sum_{r> r_0} r\mathbb E[m_r] &\leq
    \frac{2\alpha}{k}\cdot \sum_{r\geq r_0}\mathbb E[m_r]\\
    &= \frac{2\alpha}{k}\cdot\mathbb E[m_{\geq r_0}]\\
    &\leq \frac{2\alpha}{k}\cdot kn\binom{\alpha n}{r_0}(kn)^{-r_0}\\
    &\leq 2 n k^{-r_0}\\
    & \leq\eps \alpha n.
  \end{split}\]
  This concludes the proof of the lemma in second case.
\end{proof}

\subsection{The model $\mathcal G^*_{\text{$k$-out}}(G)$}\label{Section:coupling}

Theorem \ref{thm:main} provides a tool
which enables one to
find many edge-disjoint rainbow subgraphs in a typical $H\sim \mathcal G_{c}(G,p)$,
each of which is distributed almost as $\mathcal G_{k\text{-out}}(G)$. To prove this,
we use a closely related model to $\mathcal G_{k\text{-out}}(G)$, the model $\mathcal G^*_{k\text{-out}}(G)$ (as defined in Section \ref{sec:notation}).
In the following lemma we show a coupling argument that connects the models $\mathcal G_{\text{$k$-out}}(G)$ and $\mathcal G^*_{\text{$k$-out}}(G)$.

\begin{lemma}\label{lemma:star doesnt matter}
  Let $k$ be a positive integer and
  let $G$ be a graph on $n$ vertices with $\delta(G)=\omega(\log n)$. Then for every
  monotone increasing property  $\mathcal P$ of graphs we have
  \[ \Pr[H\models \mathcal P]\leq \Pr[H'\models \mathcal
  P] + n^{-\omega(1)},\] where $H\sim \mathcal G_{\text{$k$-out}}(G)$ and $H'\sim\mathcal
  G^*_{\text{$k$-out}}(G)$.
\end{lemma}

\begin{proof}
  Throughout the proof, we assume for simplicity of notation
  that the vertex set of $G$ is $[n]$.
  If $\mathcal G_1$ and $\mathcal G_2$ are probability spaces that
  depend on some integer parameter $n$, then we use the notation
  $H_1\approx \mathcal G_2$ to mean that there exists a
  joint distribution of random variables $H_1\sim \mathcal G_1$ and $H_2\sim\mathcal G_2$ such that
  $\Pr[H_1=H_2]=1-n^{-\omega(1)}$. Roughly speaking, $H_1\approx \mathcal G_2$
  means that the distribution of $H_1$ converges to $\mathcal G_2$ at a very fast rate. In this case we
  clearly  have  $\Pr[H_1\models \mathcal P] = Pr[H_2\models \mathcal
  P] \pm n^{-\omega(1)}$ for every graph property $\mathcal P$.
  
  Note that while generating a graph $H\sim \mathcal G_{\text{$k$-out}}(G)$, it might
  happen that an edge $xy\in E(G)$ has been chosen to be in $E(H)$
  by both $x$ and $y$. However, in $\mathcal G^*_{\text{$k$-out}}(G)$ this never happens.
  Therefore, as an intermediate level of our coupling,
  we introduce the random graph model $\widehat{\mathcal G}_{\text{$k$-out}}(G)$.
  In this model, a graph $H\sim \widehat{\mathcal G}_{\text{$k$-out}}(G)$ is generated
  in the following way. Let $\sigma \in S_n$ be a random permutation of the vertices.
  Then in each step $i$, the vertex $\sigma(i)$ picks $k$ distinct incident edges
  uniformly at random among all the edges in $E(G)$ that have not been picked
  in a previous step. If there are less than $k$ edges, then it just
  picks all of them (it might happen that it picks no edges at all).
  
  A moment's thought now reveals that for every monotone increasing property
  $\mathcal P$ of graphs, we have
  \[ \Pr[\mathcal G_{\text{$k$-out}}(G)\models \mathcal P]\leq \Pr[\widehat{\mathcal G}_{\text{$k$-out}}(G)\models \mathcal P].\]
  Therefore, it suffices to show that
  \[ \Pr[\widehat{\mathcal G}_{\text{$k$-out}}(G)\models \mathcal P] = \Pr[\mathcal G^*_{\text{$k$-out}}(G)\models \mathcal P] + n^{-\omega(1)}.\]
  
  In order to do so we will describe a procedure that samples a graph $H$ from some (implicit) probability space in such a way that
  $H \approx \widehat{\mathcal G}_{\text{$k$-out}}(G)$ and $H \sim \mathcal G^*_{\text{$k$-out}}(G)$.
  Before we present the exact algorithm of the procedure, we first give the following brief description.
  Starting with the graph $G$ we order the vertices according to a random permutation $\sigma \in S_n$.
  We then orient the edges of $G$, step by step.
  For every vertex $x$ let $N^\circ(x)$ be the set of vertices $y$ which are adjacent to $x$ and where
  the edge $\{x,y\}$ is currently either undirected or directed from $x$ to
  $y$. Roughly speaking, we work our way through the vertices of $G$ according to $\sigma$, and for
  every vertex $x$ we choose a random subset of $N^\circ(x)$ of
  size $k$ and then orient the corresponding edges away from $x$. We do so by first ordering $N^\circ(x)$ randomly,
  according to a uniformly random permutation $\pi_x$.
  Then we choose a set of edges $A_x$ as follows: start with $j=1$; at the $j^{th}$
  step, if the edge $\{x,\pi_x(j)\}$ is oriented $\overrightarrow{x\pi_x(j)}$, then
  we place $\{x,\pi_x(j)\}$ into $A_x$, otherwise, we orient $\{x,\pi_x(j)\}$
  with probability $0.5$ to each direction and then place $\{x,\pi_x(j)\}$ into $A_x$
  if and only if the chosen orientation is
  $\overrightarrow{x\pi_i(x)}$. We do this for $j=1,2,3,\dotsc$ until either
  $k$ out-edges have been chosen or until the neighborhood of $x$ has been
  exhausted, then we move to the next vertex.
  After having treated all vertices in this way, we define $H$ by $V(H)=V(G)$ and $E(H)=\bigcup_{x\in V(G)} A_x$.
  
  We now give a formal description of the procedure. During the procedure,
  we maintain a set $D$ of \emph{oriented} edges and sets of edges $A_x$ (for every vertex $x$).
  We also maintain sets $N^\circ(x)$ as defined above.
  The following algorithm produces a subgraph $H$ as described above,
  as well as orientation $O$ of $G$.
  
  \begin{algorithm}[H]
    \KwData{A graph $G$ on the vertex set $V(G)=[n]$.}
    \KwResult{A subgraph $H\subseteq G$ and an orientation $O$ of $G$.}
    \Begin{
      $D\gets \emptyset$\;
      $A_1,\dotsc,A_n\gets \emptyset$\;
      $N^\circ(x)\gets N_G(x)$ for every $x\in [n]$\;
      let $\sigma\colon[n]\to[n]$ be a random permutation\;
      let $\{X_{xy} \mid (x,y)\in [n]\times [n]\}$ be a family of mutually independent
      random variables such that $X_{xy}\sim \Ber(1/2)$\;
      \For{$i = 1,2,\dotsc,n$}{
        $x \gets \sigma(i)$\;
        let $\pi_{x}: [|N^\circ(x)|]\rightarrow N^\circ(x)$
      be a random ordering of $N^\circ(x)$\;
        $j\gets 1$\;
        \While{$\sum_{\ell=1}^{j-1} X_{x\pi_{x}(\ell)}< k$ and $j< |N_{\sigma(i)}|$}{
          $y\gets \pi_{x}(j)$\;
          \uIf{$\overrightarrow{xy}\in D$}{
            $A_{x} \gets A_{x}\cup \{\{x,y\}\}$\;}
          \Else{
            \uIf{$X_{xy}=1$}{
              $A_{x} \gets A_{x}\cup \{\{x,y\}\}$\;
              $N^\circ(y) \gets N^\circ(y)\setminus \{x\}$\;
              $D \gets D\cup \{\overrightarrow{xy}\}$;}
            \Else{$D \gets D\cup \{\overrightarrow{yx}\}$\;
            }}
          
          $j\gets j+1$\;
        }
      }
      let $O$ be an orientation of $G$ chosen uniformly at random
      among all orientations that are consistent with $D$\;
      let $H$ be the graph on $[n]$ with edge set $\bigcup_{x\in [n]}A_x$\;
    }
  \end{algorithm}

  Now we wish to show that the resulting graph $H$ satisfies
  $H\sim \mathcal G^*_{\text{$k$-out}}(G)$ and
  $H\approx \widehat{\mathcal G}_{\text{$k$-out}}(G)$. First, we show
  that $H\sim \mathcal G^*_{\text{$k$-out}}(G)$. To this end, note
  that the orientation $O$ generated in the procedure above is indeed
  a random orientation. In addition, note that for each $x$, the set
  of edges $A_x$ is of size $|A_x|=\min{\{k, d^+_{OG}(i)\}}$, where $OG$
  is the oriented graph obtained from $G$ equipped with $O$, and is
  chosen uniformly at random among all $|A_x|$-subsets of the set of
  edges starting at $x$ in
  $OG$. Therefore, we conclude that indeed $H\sim \mathcal
  G^*_{\text{$k$-out}}(G)$.

  Next, we show that  $H\approx \widehat{G}_{\text{$k$-out}}(G)$. For this purpose we note the following facts.
  \begin{enumerate}[\em (i)]
  \item During the procedure we consider vertices according to a random permutation $\sigma \in S_n$.
  \item In each step $i$, where we consider the vertex $x=\sigma(i)$, 
    the set $N^\circ(x)$ contains exactly the neighbors $y$ of $x$ for which
    the edge $\{y,x\}$ is still unclaimed (i.e., for which $\{y,x\}\not\in A_y$).
  \item When we consider the vertex $x$, we choose the set $A_x$ as a random
    subset of $N^\circ(x)$ of size $|A_x|=\min{\{k,d^+_{OG}(x)\}}$.
    Indeed, since $\pi_{x}$ is a random permutation on
    $N^\circ(x)$, due to symmetry, it is clear that the chosen set
    is a uniformly random subset of size $|A_x|$.
  \item By the Chernoff inequality (Lemma~\ref{lemma:chernoff}) and the union
    bound, the probability
    that we have $d^+_{OG}(x) \geq k$ for every $x\in V(G)$
    is at least
    \[ 1-n e^{-\Omega(\delta(G))} = 1-n^{-\omega(1)},\]
    since $\delta(G) = \omega(\log n)$.
  \end{enumerate}
  
  Conditioned on the event that $d_{OG}^+(x)\geq k$ holds for every $x\in V(G)$,
  the procedure generates a member of
  $\widehat{G}_{\text{$k$-out}}(G)$ by {\em (i)--(iii)}.
  Therefore by {\em (iv)} we have
  $H\approx \widehat{G}_{\text{$k$-out}}(G)$, as required.
\end{proof}

\section{Proof of Theorem \ref{thm:main}}\label{sec:MainProof}

In this section we prove the main result of this paper, Theorem \ref{thm:main}. For this aim we prove the following theorem, which together with Lemma \ref{lemma:star doesnt matter} implies Theorem \ref{thm:main}.

\begin{theorem}\label{theorem:splitting to rainbow k-out} Let $\varepsilon>0$ be a constant and  let $k\geq 2$ be an integer. Let $G$ be a graph on $n$ vertices with minimum
	degree $\delta(G)=\omega(\log n)$, and let $p\in (0,1]$ be such that $p=\omega\left(\frac{\log n}{\delta(G)}\right)$. Then,
	w.h.p.\ a graph $H\sim \mathcal G_{kn}(G,p)$ can be generated as $H=H_0\cup H_1\cup\dots\cup H_t$ such that the following holds.
	\begin{enumerate}
		[$(i)$]
		\item $H_0,H_1,\ldots,H_t$ are edge-disjoint subgraphs,
		\item $t:=\left(1-\varepsilon\right)\frac{\delta(G) p}{2k}$,
		\item for every $1\leq i\leq t$, $E(H_i)$ is rainbow, and
		\item for every $1\leq i\leq t$, $H_i\sim \mathcal G_{k\text{-out}}^*(G)$

	\end{enumerate}
\end{theorem}


\begin{proof}[Proof of Theorem \ref{theorem:splitting to rainbow k-out}]
  First, we describe a procedure that w.h.p.\ generates a graph $H\sim\mathcal G_{kn}(G,p)$, which we will use in order to find the required subgraphs $H_i$. The procedure consists of the following six steps.
  
  \begin{enumerate}[(I)]
  \item In this step we orient all the edges of $G$ at random, and denote the resulting graph by $\overrightarrow{G}$. That is, for every edge $xy\in E(G)$, choose an orientation $\overrightarrow{xy}$ or $\overrightarrow{yx}$ with probability $1/2$ independently at random. Denote by $N^+(x)$ the out-neighborhood of $x$ (that is, the set of all
    $y\in V(G)$ such that $\overrightarrow{xy}$ is an edge), and let $d^+(x):=|N^+(x)|$.
  \item For every $x\in V(G)$,
    let $D_H(x)\sim \Bin(d^+(x),p)$, denote a corresponding
    out-degree of $x$ in a random oriented subgraph of $G$,
    chosen by keeping each (oriented) edge with probability $p$, independently at random.
  \item For every $x\in V(G)$, let $C_x$ be a random multi-subset of $[kn]$ of size $s=(1-\eps/4)\delta(G)p/2$
    and let $C_x = \{c_1^x, \dotsc, c_s^x\}$ be an \textbf{arbitrary} ordering of the elements of
    $C_x$ (counted with multiplicities).
    If $s> D_H(x)$, the procedure fails.
  \item For every $x\in V(G)$, let $\sigma_x\colon [D_H(x)]\to N^+(x)$ be a random injection.
  \item For every $x\in V(G)$ and $i\in [s]$, put an edge $x\sigma_x(i)$ into $H$ and color it with the color $c_i^x$.
  \item For every $x\in V(G)$ and $s<i \leq D_H(x)$, put an edge $x\sigma_x(i)$ into $H$ and color it with a random color
    from $[kn]$.
  \end{enumerate}
  
  Note that using Chernoff's inequality (Lemma~\ref{lemma:chernoff}), the union bound, and the fact that
  $p= \omega(\log n/\delta(G))$,
  after taking a random orientation in Step (I),
  w.h.p.\ we have $d^+(x)\geq  (1-o(1))\delta(G)/2$ for every $x\in V(G)$. Furthermore, in Step (II),  w.h.p.\ we have
  $D_H(x) = (1-o(1))d^+(x)p$ for every $x\in V(G)$.
  Therefore, the procedure only fails with a negligible probability. It is clear that if the procedure
  succeeds, then it outputs a graph $H$ distributed as $\mathcal G_{kn}(G,p)$.
  
  Note that in step (III) we can choose any ordering of $C_x$. In order to prove our theorem,
  we show that w.h.p.\ one can choose the orderings so that the following holds: there exists a $t\geq
  (1-\eps)\delta(G)p/2$
  which is divisible by $k$ such that for every $0\leq i< t/k$,
  $C(i) := \{c_{ki}^x,\dotsc,c_{k(i+1)-1}^x\}_{ x\in V(G)}$ is a set of size $kn$ (i.e., all the elements of $C(i)$ are distinct).
  Assuming we are able to choose the orderings of $C_x$ as described, then for every $0\leq i \leq t/k$, let
  $E_i=\{x\sigma_x(ki), \dotsc, x\sigma_x(k(i+1)-1)\}$ be all the edges that get assigned the colors from $C(i)$
  in step (V) and define the graph $H_i=(V(G),E_i)$. Since in the above procedure all the edges are chosen randomly, we observe that $H_i\sim \mathcal G^*_{k\text{-out}}(G)$ and that $E_i$ is rainbow for each $i$. This then implies the theorem.
  
  In order to find such an ordering, we will act as follows.
  Let $r_0=r_0(k,\eps)$ be as in Lemma~\ref{lemma:concentration}, applied to $\eps/4$ (as $\eps$) and $\alpha = (1-\eps/4)\delta(G)p/(2 n)$. For each $x\in V(G)$, let $C_x$ be a random multi-subset of $[kn]$ of size $s$ (recall
  that $s=(1-\eps/4)\delta(G)p/2 = \alpha n$). For every $r\in [s]$, let $C_x^r$ denote the set of elements in $C_x$ with
  multiplicity exactly $r$ and let $m_r^x := |C_x^r|$. Note that for $x\neq y$, $m_r^x$ and $m_r^y$ are identically
  distributed, so $\mathbb E[m_r^x] = \mathbb E[m_r^y] =: \mu_r$.
  For every $r\in[r_0]$, let $d_r := (1-\eps/4)\mu_r$. Observe that by Lemma~\ref{lemma:concentration} and the union
  bound, w.h.p.\ for every $x\in V(G)$ and $r\in [r_0]$, we have $m_r^x\geq d_r$ and $d_r = \omega(\log n)$.
  
  Now, for each $r\in [r_0]$,
  define an auxiliary bipartite graph $B_r$ in the following way. Let $L_r = V(G)$ and $R_r = [kn]$ be the parts
  of $B_r$ and for $x\in L_r$ and $c\in R_r$, $xc\in E[B_r]$ if and only if $c\in C_x^r$. That is, every edge $xc$ of $B_r$ corresponds to a vertex $x\in V(G)$ and to a color $c$ which appears in $C_x$ with multiplicity exactly $r$.
  Roughly speaking, our goal now is to find many edge-disjoint $k$-matchings in each of these graphs.
  Then, from each such matching we take $r$ copies. Note that each such copy can play the role of a $C(i)$.
  
  In order to find edge-disjoint $k$-matchings in each of the $B_r$, observe
  that by taking $d_r$ elements at random from each set $C_x^r$ (assuming that $d_r\leq m_r^x$), we obtain a subgraph $B_r' \sim \mathcal B^\ell_{d_r\text{-out}}(n,kn)$. Recall that $m_r^x<d_r$ happens with negligible probability. Therefore, by Corollary~\ref{cor:star-factors}, w.h.p.\ for every $r\in [r_0]$,
  the graph $B_r$ contains $s_r = (1-\eps/4)d_r/k$ edge-disjoint perfect $k$-matchings $\{M_{r,j} \mid j\in [s_r]\}$.
    
  Next, we wish to define the multi-sets $C_x$ by taking $r$ copies of each $M_{r,j}$ and arranging it in a ``matrix-like" structure, where the rows correspond to the multi-sets $C_x$ in such a way that ``$k$-blocks" of columns correspond to the matchings $M_{r,j}$. Formally, we define the ordering of each $C_x$ as follows. For each $r\in [r_0]$ and $j\in [s_r]$ and for each vertex $x\in L_r$ let $N_{r,j}(x)$ denote
  the neighbors of $x$ in $M_{r,j}$.
  
  Note that for each $r\in [r_0]$,
  \[ks_r = (1-\eps/4)d_r = (1-\eps/4)^2 \mu_r\geq (1-\eps/2)\mu_r.\]
  Moreover, by Lemma~\ref{lemma:concentration}, we have
  $\sum_{r\in [r_0]}r \mu_r\geq (1-\eps/4)s$.
  Therefore, for $t := k\sum_{r\in [r_0]}rs_r$, we have $t\geq (1-3\eps/4)s\geq (1-\eps)\delta(G)p/2$.
  Note that $t$ is a multiple of $k$ and $t\leq s$.
  
  Let $x\in V(G)$ and $r\in [r_0]$.
  For each $j\in [s_r]$, let $\pi_{r,j}^x\colon [k]\to N_{r,j}(x)$ be
  an arbitrary ordering of $N_{r,j}(x)$.
  Now we construct the ordering of $C_x$ iteratively as follows.
  For $r=1$, let
  \[ A_x^1 = \pi^x_{1,1}(1),\dotsc,\pi^x_{1,1}(k),\pi^x_{1,2}(1),\dotsc,\pi^x_{1,2}(k),\dotsc,\pi^x_{1,s_1}(1),\dotsc,\pi^x_{1,s_1}(k).\]
  Now assume that we have constructed $A_x^{r-1}$, then we construct $A_x^r$ by taking
  $r$ copies of each $N_{r,j}(x)$ in the following way:
  \[ A_x^{r} = A_{x}^{r-1}, \underset{r}{\underbrace{a_x^r ,\dots,a_x^r}}\]
  where
  \[a_x^{r}:=\pi^x_{r,1}(1),\dotsc,\pi^x_{r,1}(k),\pi^x_{r,2}(1),\dotsc,\pi^x_{r,2}(k),\dotsc,\pi^x_{r,s_{r}}(1),\dotsc,\pi^x_{r,s_{r}}(k).\]

  Finally, let $C'_x := C_x\setminus A_x^{r_0}$
  be the set of remaining neighbors (note that here we refer to $A_x^{r_0}$ as a multi-set and not a list) and
  order it arbitrarily.
  The ordering of $C_x$ is then given by
  first putting $A_x^{r_0}$ and then the ordering of $C'_x$. Let
  $c_{1}^x,\dotsc,c_{s}^x$ be the ordering obtained in this way.
  Note that for each $0\leq i < t/k$,
  the set $C(i) = \{c_{ki}^x,\dotsc,c_{k(i+1)-1}^x \}_{x\in V(G)}$ corresponds to the right side
  of some $M_{r,j}$ and therefore contains $kn$ distinct colors.
  
  This completes the proof.
\end{proof}

\section{Applications of Theorem \ref{thm:main}}\label{sec:apps}

In this section we prove Theorems \ref{thm:PM} and \ref{thm:Ham}. In order to prove Theorem \ref{thm:PM} we make use of the following lemma due to Walkup~\cite{walkup1980matchings}.
\begin{lemma}\label{lemma:pm in 3-out}
  Let $G\sim \mathcal B_{\text{3-out}}(K_{n,n})$. Then the probability that $G$ contains a perfect matching
  is at least $1-n^{-4}$.
\end{lemma}

Now we are ready to prove Theorem \ref{thm:PM}.

\begin{proof}[Proof of Theorem \ref{thm:PM}]
  Observe that containing a perfect matching is a monotone
  graph property. Apply Theorem~\ref{thm:main} with $k=3$ and $G=K_{n,n}$. Thus, we can w.h.p.\ split
  the graph $H\sim \mathcal B _{3n}(n,n,p)$ into $t=(1-o(1))\frac{np}{6}$ edge-disjoint rainbow
  subgraphs $H_1,\dotsc,H_t$, such that for each $i$, we have
  \[\Pr[H_i\text{ has a perfect matching}] \geq \Pr[\mathcal G_{\text{3-out}}(n)\text{ has a perfect matching}]-n^{-\omega(1)} \geq 1-o(1/t), \]
  using Lemma~\ref{lemma:pm in 3-out} for the last inequality.
  By the union bound,
  we get that w.h.p.\ $H$ contains $t$ edge-disjoint rainbow perfect matchings. This completes the proof.
\end{proof}

For the packing of Hamilton cycles, we use the following version of a theorem of Fenner
and Frieze \cite{fenner1983existence}.
\begin{lemma}[Theorem 2.7 from \cite{fenner1983existence}]\label{lemma:ham}
  Let $H \sim \mathcal{G}_{\text{23-out}}(n)$. Then the probability that $H$ contains a Hamilton cycle is $1-o(\frac{1}{n})$.
\end{lemma}

We note that in \cite{fenner1983existence} the error term is not explicitly given as $1-o(1/n)$, but this can be obtained by a more careful analysis of an earlier result from \cite{FF84}, which has been used there. Another approach would be to use a well known result about 3-out from~\cite{bohman2009hamilton}. With this, after a lengthy calculation, one could achieve a better constant in Theorem~\ref{thm:Ham}, but still nothing asymptotically optimal.

\begin{proof}[Proof of Theorem~\ref{thm:Ham}]
  Let $\mathcal P$ be the monotone property of containing a Hamilton cycle. Let $H\sim\mathcal G_{23n}(n,p)$,
  where $p=\omega(\log n/n)$.
  As before we apply Theorem \ref{thm:main} with $k=23$ and $G=K_n$ and get $(1-\varepsilon)\frac{np}{46}$ rainbow subgraphs $H_i$ of $H$ such that for every $i$, we have
  \[ \Pr[H_i \models \mathcal P] \geq \Pr[\mathcal G_{\text{23-out}}(n)\models P] - n^{-\omega(1)}
\geq 1-o(1/t),\]
using Lemma~\ref{lemma:ham} for the last inequality.
Then the union bound implies that $H$ contains a family of $(1-\varepsilon)\frac{np}{46}$ rainbow Hamilton cycles.
\end{proof}

Finally, to prove Theorem~\ref{thm:Dirac-Ham}, we need the following result of
Frieze and Johansson~\cite{frieze2014random}. Again, the probability for failure is not explicitly given
as $o(1/n)$ in~\cite{frieze2014random}, but is easily verified to be so.
\begin{lemma}[Theorem 2 from~\cite{frieze2014random}]\label{lemma:frieze}
  For every $\eps>0$, there exists a positive integer $k_{\eps}$ such that the following holds.
  Let $G$ be a graph on $n$ vertices with minimum degree $\delta(G)\geq (1+\varepsilon)n/2$. Then
  the probability that $H\sim \mathcal G_{k_{\eps}\text{-out}}(G)$ is not Hamiltonian is $o(1/n)$.
\end{lemma}

\begin{proof}[Proof of Theorem~\ref{thm:Dirac-Ham}]
  Let $\eps>0$ and let $k_{\eps}$ be a sufficiently large integer. Let $\mathcal P$ be the monotone property of being
  Hamiltonian. Assume that $H\sim \mathcal G(G,p)$, where $p=\omega(\log n/n)$. Then, by Theorem~\ref{thm:main},
  we see that $H$ splits into $t=(1-o(1))np/(2k_\eps)$ graphs $H_1,\dotsc,H_t$ such that each $H_i$ is
  rainbow and
  \[ \Pr[H_i\models \mathcal P] \geq \Pr[\mathcal G_{k_\eps\text{-out}}(G) \models \mathcal P] - n^{-\omega(1)}. \]
  By Lemma~\ref{lemma:frieze}, we have $\Pr[\mathcal G_{k_\eps\text{-out}}(G) \models \mathcal P] = 1-o(1/n)$,
  and so the union bound immediately gives that w.h.p.\ each of the $H_i$ is Hamiltonian, completing the proof.
\end{proof}

%
%
%


%
%
%
%


\bibliographystyle{abbrv}
\bibliography{references}

\end{document}